\documentclass[12pt]{article}
\usepackage{latexsym,amsmath,amsthm,amssymb,amscd,amsfonts}
\usepackage[usenames]{color}

\newtheorem{lemma}{Lemma}[section]

\newtheorem{remark}[lemma]{Remark}
\newtheorem{theorem}[lemma]{Theorem}
\newtheorem{definition}[lemma]{Definition}
\newtheorem{claim}[lemma]{Claim}
\newtheorem{observation}[lemma]{Observation}

\newtheorem{corollary}[lemma]{Corollary}

\newtheorem*{remark*}{Remark}

\makeatletter \@addtoreset {equation}{section}

\renewcommand\theequation
  {\ifnum \c@subsection>\z@ \arabic{section}.\arabic{subsection}.\arabic{equation}
  \else \arabic{section}.\arabic{equation} \fi}
\makeatother

\newcommand{\RNum}[1]{\uppercase\expandafter{\romannumeral #1\relax}}

\begin{document}

\title{On Fabry's quotient theorem}

\author{Lev Buhovsky}

%\footnotetext[1]{The author also uses the spelling ``Buhovski"
%for his family name.}

\date{\today}
\maketitle

\begin{abstract}
We present a short proof of the Fabry quotient theorem, which states that for a complex power series with unit radius of convergence, if the quotient of its consecutive coefficients tends to $ s $, then the point $ z=s $ is a singular point of the series. This proof only uses material from undergraduate university studies. 
\end{abstract}

\section{Introduction}

Fabry's celebrated theorems detect singular points of power series on the boundary of the disc of convergence and provide large classes of Taylor series, which cannot be analytically continued through any arc of the boundary circle. The reader will find these theorems in the Dienes and Bieberbach treatises \cite[Sections~93-94]{D}, \cite[Chapter 2]{B}. For more recent results and references, see for instance, papers by Arakelian and Martirosyan \cite{AM1,AM2}, by Arakelian, Luh and M\"uller \cite{ALM}, and by Eremenko \cite{Erem1,Erem2}. 

Fabry's theorems are known for their formidable formulations and complicated proofs. The original Fabry's proofs were quite ingenious and long but used only basic properties of Taylor series \cite[Section 2.1]{B}. Faber and then P\'olya developed another approach to general Fabry theorems which is based on the interpolation of the coefficients of the Taylor series by an entire function and on connection between growth of entire functions and distribution of their zeroes. This approach is well explained in the mentioned above books by Dienes and Bieberbach. 

Probably, the most known consequence of general Fabry's theorems is {\em Fabry's gap theorem}, which has numerous connections with other areas of analysis and has attracted attention of many prominent mathematicians, see for instance, \cite[Section 32]{PW}, \cite[Chapter 2]{B}, \cite[Chapter \RNum{12}]{E}, \cite[Sections \RNum{2}.4.6-\RNum{2}.4.10]{HJ}.

Another remarkable consequence of a general Fabry theorem is {\em Fabry's quotient theorem} whose elegant formulation deserves to be included into the courses of basic complex analysis:
\begin{theorem} \label{thm:fabry-quotient}
Let $ f(z) = \sum_{n=0}^\infty a_n z^n $ be a complex power series with unit radius of convergence.  Assume that $ \lim_{n \rightarrow \infty} \frac{a_n}{a_{n+1}} = s $. Then $ z = s $ is a singular point of $ f $.
\end{theorem} 

My friend and colleague Misha Sodin recently mentioned to me that he is not aware of a proof of this theorem which might be explained in the first course of complex analysis or included in textbooks. In this short note we provide such a proof, which is based on an idea from harmonic analysis. A somewhat similar approach was used by Wiener for proving a version of Fabry's gap theorem which was weaker than the original one \cite{W}, \cite[Section 32]{PW}, as well as in other instances, see also \cite{S}, \cite[Lemma]{EF}, \cite[Chapter 7, Section 3]{M} (and the Notes to this Chapter). 

It could be that our proof of Fabry's Theorem \ref{thm:fabry-quotient} (or its another short and elementary proof) might be known to experts, or even published, although we didn't find any evidence to that.

\paragraph{Acknowledgements}  I thank A. Borichev, A. Eremenko, D. Khavinson, F. Nazarov and M. Sodin for a valuable feedback. I especially thank F. Nazarov for an important suggestion which led to an improvement of the actual result we show (see the quantitative statement at the beginning of our proof of Theorem \ref{thm:fabry-quotient} in Section \ref{s:proofs}) and to simplification of our original arguments, and M. Sodin for introducing me to Fabry theorems, encouraging me to write this note, and for his generous help with the presentation. This work was partially supported by ERC Starting Grant 757585 and ISF Grant 2026/17.

\section{Proofs} \label{s:proofs}

\begin{definition}
Let $ F(\theta) = \sum_{n=-\infty}^{\infty} a_n e^{in\theta} $ be a formal Fourier series, $ a_n = r_n e^{i \phi_n} $, $ r_n \geqslant 0 $, $ \phi_n \in \mathbb{R} $. Let $ N \geqslant 2 $ be an integer. We say that $ F $ is $ N $-good, if there is a choice of $ (\phi_n)_{n \in \mathbb{Z}} $, such that for each $ n \in \mathbb{Z} $ there exists some $ \Phi_n \in \mathbb{R} $ with $ \phi_n, \phi_{n+1}, \ldots, \phi_{n+N-1} \in [\Phi_n, \Phi_n + \frac{\pi}{2}] $. Similarly, for a formal Taylor series $ f(z) = \sum_{n=0}^{\infty} a_n z^n $ and an integer $ N \geqslant 2 $, we say that $ f $ is $ N $-good if one can write $ a_n = r_n e^{i \phi_n} $, $ r_n \geqslant 0 $, $ \phi_n \in \mathbb{R} $, such that for each $ n \geqslant 0 $ there exists some $ \Phi_n \in \mathbb{R} $ with $ \phi_n, \phi_{n+1}, \ldots, \phi_{n+N-1} \in [\Phi_n, \Phi_n + \frac{\pi}{2}] $. 
\end{definition}

For a trigonometric polynomial $ P(\theta) = \sum_{k = -M}^M c_k e^{ik\theta} $, we will say that $ P $ is symmetric with non-negative coefficients, if $ c_k = c_{-k} \geqslant 0 $ for each $ 0 \leqslant k \leqslant M $.

\newpage

\begin{observation} \label{obs:main}
Let $ F(\theta) =  \sum_{n=-\infty}^{\infty} a_n e^{in\theta} $ be an $ N $-good $ L^2[-\pi,\pi] $ Fourier series, and let $ P(\theta) = \sum_{k = -N+1}^{N-1} c_k e^{ik\theta} $ be symmetric with non-negative coefficients. Then 
\begin{equation}
\begin{gathered}
\int_{-\pi}^\pi |F(\theta)|^2 P(\theta) \, d\theta =  2\pi \sum_{n=-\infty}^\infty \sum_{k=-N+1}^{N-1} a_n \overline{a}_{n+k} c_k  \\ 
= 2\pi c_0 \sum_{n=-\infty}^\infty |a_n|^2 + 4\pi \sum_{n=-\infty}^\infty \sum_{k=1}^{N-1}  c_k \Re (a_n \overline{a}_{n+k}) \geqslant 0.
\end{gathered}
\end{equation}
\end{observation}

\begin{claim} \label{claim:existence-trigpol}
For every integer $ N \geqslant 2 $ there exists a trigonometric polynomial $ P(\theta) = \sum_{k = -N+1}^{N-1} c_k e^{ik\theta} ,$ symmetric with non-negative coefficients, such that we have $ P(\theta)  < 0 $ for $ \theta \in [-\pi,\pi] \setminus (-\frac{4\pi}{N},\frac{4\pi}{N}) $.
\end{claim}

We believe that Claim \ref{claim:existence-trigpol} (or its sharper versions) is well known to experts. We postpone our proof of the claim to the end of the section. The combination of Observation \ref{obs:main} and Claim \ref{claim:existence-trigpol} implies the following corollary:

\begin{corollary} \label{cor:main}
Let $ F(\theta) =  \sum_{n=-\infty}^{\infty} a_n e^{in\theta} $ be an $ N $-good $ L^2[-\pi,\pi] $ Fourier series. Then
\begin{equation} \label{eq:L2-main}
\int_{-\pi}^\pi |F(\theta)|^2 \, d\theta \leqslant C \int_{-\frac{4\pi}{N}}^{\frac{4\pi}{N}} |F(\theta)|^2 \, d\theta, 
\end{equation}
where $ C = C(N) $.
\end{corollary}

\begin{remark} \label{rem:Wiener}
An analogous to $(\ref{eq:L2-main})$ inequality was obtained by N. Wiener in \cite{W} for Fourier series having ``uniform gaps'', with a sharp growth rate of the corresponding coefficient $ C $ (see Theorem I therein).  In addition, Wiener showed an analogous to $(\ref{eq:L2-main})$ inequality for Fourier series with non-negative coefficients (see Theorem 2.1 in \cite{S}), again with a sharp growth rate of the corresponding coefficient. It would be also interesting to find the sharp growth rate of $ C(N) $ in terms of $ N $ in Corollary \ref{cor:main}. Our proof gives $ C(N) = O(N^2) $.
\end{remark}

Our proof of Theorem \ref{thm:fabry-quotient} uses the following simple claim:

\begin{claim} \label{claim:Poincare}
If $ v : \mathbb{R} \rightarrow \mathbb{C} $ is a $ 2\pi $-periodic continuously differentiable function having zero mean on $ [-\pi,\pi] $ (i.e. $ \int_{-\pi}^{\pi} v(t) \, dt = 0 $), then $ \max_{t \in [-\pi,\pi]} |v(t)| \leqslant \sqrt{ \frac{\pi}{2} \int_{-\pi}^{\pi} |v'(t)|^2 \, dt } $. 
\end{claim}
\begin{proof}
Let us first show the inequality
\begin{equation} \label{eq:geom}
\max_{t \in [-\pi,\pi]} |v(t)| \leqslant \frac{1}{2} \int_{-\pi}^{\pi} |v'(t)| \, dt. 
\end{equation}
Given $(\ref{eq:geom}) $, the statement of the claim readily follows by the Cauchy-Schwartz inequality. Note, that the inequality $(\ref{eq:geom})$ has a simple geometric meaning: given a continuously differentiable curve on the plane $ \mathbb{R}^2 $, the distance from its mean value to any point on the curve is not larger than half of the length of the curve. 
%To see why $(\ref{eq:geom})$ is true, notice that the distance between any two points on the curve is not larger than half of the length of the curve (since these points divide the curve in two, one of which has length smaller at least by half). Then by fixing the first point, and averaging the inequality over the second one, we get $(\ref{eq:geom})$. 

%\footnote{Here, instead of the {\em mean value}, one could consider the {\em center of mass} of $ v $ with respect to an arbitrary mass distribution over the curve (this could be seen from the independence of the inequality $(\ref{eq:geom})$ from a choice of parametrisation of $ v $, or alternatively, the current proof could be readily adapted to this more general setting). Equivalently, $(\ref{eq:geom})$ holds for any continuously differentiable curve $ v $ on the plane $ \mathbb{R}^2 $ such that the origin lies in the convex hull of $ v $.}

 For a proof of $(\ref{eq:geom})$, let us first show that for any $ t, T \in [-\pi,\pi] $ we have 
\begin{equation} \label{eq:geomprelim}
|v(T)-v(t)| \leqslant \frac{1}{2} \int_{-\pi}^{\pi} |v'(s)| \, ds.  
\end{equation}
Indeed, we can WLOG assume that $ -\pi \leqslant t \leqslant T \leqslant \pi $, and then 
\begin{equation*}
\begin{gathered}
|v(T) - v(t)| = \frac{1}{2} \left| \int_t^T v'(s) \, ds - \int_T^{t+2\pi} v'(s) \, ds \right| \\
\leqslant \frac{1}{2} \left( \int_t^T |v'(s)| \, ds + \int_T^{t+2\pi} |v'(s)| \, ds \right) = \frac{1}{2} \int_{-\pi}^{\pi} |v'(s)| \, ds. 
\end{gathered}
\end{equation*}
Now, having $(\ref{eq:geomprelim})$, we conclude that for each $ T \in [-\pi,\pi] $, 
\begin{equation*}
\begin{gathered}
|v(T)| = \left| \frac{1}{2\pi} \int_{-\pi}^{\pi} v(T) - v(t) \, dt \right| \leqslant  \frac{1}{2\pi} \int_{-\pi}^{\pi}  |v(T) - v(t)| \, dt %\\
\leqslant  \frac{1}{2} \int_{-\pi}^{\pi} |v'(s)| \, ds, 
\end{gathered}
\end{equation*}
and $(\ref{eq:geom})$ follows.
\end{proof}

We now pass to the proof of Theorem \ref{thm:fabry-quotient}.

\begin{proof}[Proof of Theorem \ref{thm:fabry-quotient}]

We will in fact prove a more general quantitative statement: If $ f(z) = \sum_{n=0}^\infty a_n z^n $ is an $ N $-good Taylor series with radius of convergence $ R=1 $, then $ f $ cannot be extended analytically through the arc $ \mathcal{C} = \{ e^{i\theta} \, | \, \theta \in [-\frac{4\pi}{N},\frac{4\pi}{N}] \} $, i.e. $ f $ is not analytic on $ \{ |z| < 1 \} \cup \mathcal{C} $. 
%We will in fact prove a more general quantitative statement: If $ f(z) = \sum_{n=0}^\infty a_n z^n $ is an $ N $-good Taylor series which is analytic in $ \{ |z| < 1 \} \cup \mathcal{C} $, where $ \mathcal{C} = \{ e^{i\theta} \, | \, \theta \in [-\frac{4\pi}{N},\frac{4\pi}{N}] \} $, then $ f $ is analytic in the closed disc $ \{ |z| \leqslant 1 \} $ (i.e. the radius of convergence of $ f(z) = \sum_{n=0}^\infty a_n z^n $ is greater than $ 1 $).

To deduce the theorem from this statement, first notice that in the theorem one can WLOG assume that $ s = 1 $, by change of variables $ w = z/s $. Then, the assumption of the theorem gives us $ \lim_{n \rightarrow \infty} \frac{a_n}{a_{n+1}} = 1 $, and in particular, for every integer $ N \geqslant 2 $ there exists $ m \in \mathbb{N} $ such that $ f_m(z) := \sum_{n=m}^\infty a_n z^n $ is $ N $-good and has radius of convergence $ R = 1 $. Therefore, assuming the statement, we conclude that $ f_m $, and hence $ f $ as well, cannot be extended analytically through the arc $ \mathcal{C} = \{ e^{i\theta} \, | \, \theta \in [-\frac{4\pi}{N},\frac{4\pi}{N}] \} $. Since this holds for every integer $ N \geqslant 2 $, the theorem follows.

Let us now pass to our proof of the above quantitative statement. The proof is by contradiction, so assume that we have an $ N $-good Taylor series $ f(z) = \sum_{n=0}^\infty a_n z^n $ with radius of convergence $ R=1 $, such that $ f $ is holomorphic on $ \{ |z| < 1 \} \cup \mathcal{C} $, where $ \mathcal{C} = \{ e^{i\theta} \, | \, \theta \in [-\frac{4\pi}{N},\frac{4\pi}{N}] \} $.

Consider the $ 2\pi $-periodic complex function $ F(w) := f(e^{iw}) $. Since $ f $ is analytic in a neighbourhood of $ \{ |z| < 1 \} \cup \mathcal{C} $ (where $ \mathcal{C} = \{ e^{i\theta} \, | \, \theta \in [-\frac{4\pi}{N},\frac{4\pi}{N}] \} $ as before), we conclude that $ F $ is analytic in some neighbourhood $ \Omega $ of $ \{ \Im w > 0 \} \cup [-\frac{4\pi}{N},\frac{4\pi}{N}] $. For $ w \in [-\frac{4\pi}{N},\frac{4\pi}{N}] \times [0,\infty] \subset \Omega $, the distance from $ w $ to the boundary of $ \Omega $ is bounded from below by a positive constant. Therefore, by the Cauchy estimates we have $ |F^{(\ell)}(w)| \leqslant C_1 A^\ell \ell! $ for any $ w \in [-\frac{4\pi}{N},\frac{4\pi}{N}] \times [0,\infty] $ and $ \ell \geqslant 0 $. Now, for $ \tau > 0 $, define $ F_\tau(\theta) := F(\theta + i\tau) = f(e^{-\tau + i\theta}) $, $ \theta \in \mathbb{R} $. We conclude 
\begin{equation} \label{eq:Cauchy-est}
|F_\tau^{(\ell)}(\theta)| = |F^{(\ell)}(\theta + i\tau)| \leqslant C_1 A^\ell \ell! 
\end{equation}
for $ \theta \in [-\frac{4\pi}{N},\frac{4\pi}{N}] $ and $ \tau > 0 $.

For any given $ \tau > 0 $, the function $ F_\tau $ admits the Fourier series $ F_\tau(\theta) = f(e^{-\tau + i\theta}) = \sum_{n=0}^\infty a_n e^{-n\tau} e^{in\theta} $, which is $ N $-good. Pick some integer $ \ell \geqslant 1 $, and apply Corollary \ref{cor:main} to the Fourier series $ F_\tau^{(\ell)}(\theta) = i^\ell \sum_{n=0}^\infty n^\ell a_n e^{-n\tau} e^{in\theta} $ (which is $ N $-good as well). Hence $ F_\tau^{(\ell)} $ satisfies the inequality $(\ref{eq:L2-main})$. By the estimate $(\ref{eq:Cauchy-est})$ we conclude
\begin{equation*} 
\int_{-\pi}^\pi |F_\tau^{(\ell)}(\theta)|^2 \, d\theta \leqslant C \int_{-\frac{4\pi}{N}}^{\frac{4\pi}{N}} |F_\tau^{(\ell)}(\theta)|^2 \, d\theta \leqslant (C_2 A^\ell \ell!)^2 ,
\end{equation*}
for $ \tau > 0 $. Since $ F_\tau^{(\ell)} $ has a zero mean on $ [-\pi,\pi] $, applying Claim \ref{claim:Poincare} with $ v = F_\tau^{(\ell)} $, we get $$ \max_{\theta \in [-\pi,\pi]} |F_\tau^{(\ell)}(\theta)| \leqslant \sqrt{\frac{\pi}{2} \int_{-\pi}^\pi |F_\tau^{(\ell + 1)}(\theta)|^2 \, d\theta} \leqslant \sqrt{\frac{\pi}{2}} C_2 A^{\ell+1} (\ell+1)! 
\leqslant C_3 (2A)^\ell \ell! $$ for every $ \tau > 0 $. In other words, we have the estimate $ |F^{(\ell)}(w)| \leqslant C_3 (2A)^\ell \ell! $ for $ \Im w > 0 $ and $ \ell \geqslant 1 $. Consequently, $ F(w) $ admits an analytic continuation from the upper half-plane $ \{ \Im w > 0 \} $ through its boundary $ \partial \{ \Im w > 0 \} = \mathbb{R} $, hence $ f (z) $ is holomorphic in the closed disc $ \{ |z| \leqslant 1 \} $, contradiction.
\end{proof}

\begin{remark}
In the quantitative statement appearing at the beginning of the proof of Theorem \ref{thm:fabry-quotient}, one can clearly weaken the assumption of $N$-goodness for $  f(z) = \sum_{n=0}^\infty a_n z^n $ to an ``eventual $N$-goodness", i.e. assuming that we have $ a_n = r_n e^{i \phi_n} $, $ r_n \geqslant 0 $, $ \phi_n \in \mathbb{R} $, such that for {\em sufficiently large} $ n $ there exists some $ \Phi_n \in \mathbb{R} $ with $ \phi_n, \phi_{n+1}, \ldots, \phi_{n+N-1} \in [\Phi_n, \Phi_n + \frac{\pi}{2}] $. Moreover, the quantitative statement is not sharp, and we expect that under the assumption that $ f(z) = \sum_{n=0}^\infty a_n z^n $ has unit radius of convergence and is $ N $-good, one can conclude that $ f $ cannot be extended analytically through the shorter arc $ \mathcal{C}' = \{ e^{i\theta} \, | \, \theta \in [-\frac{\pi}{2N-2},\frac{\pi}{2N-2}] \} $. 

In fact, it is natural to consider the following more general situation. Let $ \alpha \in (0,\pi) $, and let $ f(z) = \sum_{n=0}^\infty a_n z^n $ be a Taylor series with radius of convergence $ R = 1 $, whose coefficients can be written as $ a_n = r_n e^{i \phi_n} $, $ r_n \geqslant 0 $, $ \phi_n \in \mathbb{R} $, such that for sufficiently large $ n $ there exists some $ \Phi_n \in \mathbb{R} $ with $ \phi_n, \phi_{n+1}, \ldots, \phi_{n+N-1} \in [\Phi_n, \Phi_n + \alpha] $. Then we expect that $ f $ cannot be extended analytically through the arc $ \{ e^{i\theta} \, | \, \theta \in [-\frac{\alpha}{N-1},\frac{\alpha}{N-1}] \} $ (cf. Delange's theorem \cite[Theorem 2.3.3]{B}, \cite{De}). 
\end{remark}

\newpage
It remains to prove Claim \ref{claim:existence-trigpol}.

\begin{proof}[Proof of Claim \ref{claim:existence-trigpol}]
First of all, we can WLOG assume that $ N \geqslant 8 $ (for $ 2 \leqslant N \leqslant 7 $ just take $ P(\theta) = \cos \theta = \frac{1}{2}e^{-i\theta} + \frac{1}{2}e^{i\theta} $). Consider the Fej\'er kernel $$ F_N(\theta) = \sum_{k=-N+1}^{N-1} \left( 1-\frac{|k|}{N} \right) e^{ik\theta} = \frac{1}{N} \left( \frac{\sin \frac{N\theta}{2}}{\sin \frac{\theta}{2}} \right)^2 .$$ Define $$ G_N(\theta) := F_N ( \theta + \frac{\pi}{2N} ) + F_N(\theta - \frac{\pi}{2N}) = \sum_{k=-N+1}^{N-1} \left( 2-\frac{2|k|}{N} \right) (\cos \frac{k\pi}{2N}) e^{ik\theta} .$$ We have $$  G_N(\theta) = \frac{1}{N} \left( \left( \frac{\sin (\frac{N\theta}{2} - \frac{\pi}{4})}{\sin ( \frac{\theta}{2} - \frac{\pi}{4N})} \right)^2 +  \left( \frac{\sin (\frac{N\theta}{2} + \frac{\pi}{4})}{\sin ( \frac{\theta}{2} + \frac{\pi}{4N})} \right)^2 \right) .$$ Hence 
\begin{equation*}
\begin{gathered}
A_N(\theta) := N^{-1} \left( \max \left( |\sin ( \frac{\theta}{2} - \frac{\pi}{4N})|, |\sin (  \frac{\theta}{2} + \frac{\pi}{4N})| \right) \right)^{-2} \\
 \leqslant G_N (\theta)  \\
\leqslant B_N(\theta) := N^{-1} \left( \min \left( |\sin ( \frac{\theta}{2} - \frac{\pi}{4N})|, |\sin ( \frac{\theta}{2} + \frac{\pi}{4N})| \right) \right)^{-2}
\end{gathered}
\end{equation*}

Let us show that $ G_N(\theta) < G_{\left[ \frac{N}{4} \right]} (\theta) $ for $ \theta \in [-\pi,\pi] \setminus (-\frac{4\pi}{N},\frac{4\pi}{N}) $. For this,
it is enough to show that $ B_N(\theta) < A_{\left[ \frac{N}{4} \right]} (\theta) $ for $ \theta \in [\frac{4\pi}{N},\pi] $. The latter inequality immediately follows from

\begin{equation} \label{eq:denom}
\begin{gathered}
C_N(\theta) := 2 \min \left( \sin ( \frac{\theta}{2} - \frac{\pi}{4N}), \sin ( \frac{\theta}{2} + \frac{\pi}{4N}) \right)  \\
> D_N(\theta) := \max \left( \sin ( \frac{\theta}{2} - \frac{\pi}{4\left[ \frac{N}{4} \right]}), \sin ( \frac{\theta}{2} + \frac{\pi}{4\left[ \frac{N}{4} \right]}) \right).
\end{gathered}
\end{equation}

Let us show the inequality $(\ref{eq:denom})$ for $ \theta \in [\frac{4\pi}{N},\pi] $. If $ \theta \in [\frac{\pi}{2},\pi] $ then $ \frac{\theta}{2} \pm \frac{\pi}{4N} \in (\frac{\pi}{6},\frac{5\pi}{6}) $ and hence $ C_N(\theta) > 1 \geqslant D_N(\theta) $. Finally, if $ \theta \in [\frac{4\pi}{N}, \frac{\pi}{2}] $ then 

\begin{equation*}
\begin{gathered}
C_N(\theta) = 2\sin(\frac{\theta}{2}-\frac{\pi}{4N}) \geqslant 2\sin(\frac{\theta}{2}-\frac{\pi}{4N}) \cos(\frac{\theta}{2}-\frac{\pi}{4N}) \\
= \sin(\theta-\frac{\pi}{2N}) > \sin(\frac{\theta}{2}+\frac{\pi}{4\left[ \frac{N}{4} \right]}) = D_N(\theta).
\end{gathered}
\end{equation*}

Now put $ P(\theta) := G_N(\theta) - G_{\left[ \frac{N}{4} \right]} (\theta) $. Then $ P(\theta) $ has all the needed properties (also note that $ P(\theta) $ has a zero mean on $ [-\pi,\pi] $).
\end{proof}

\bigskip
\noindent Lev Buhovski\\
School of Mathematical Sciences, Tel Aviv University \\
{\it e-mail}: levbuh@tauex.tau.ac.il
\bigskip

\end{document}